\documentclass[a4paper,english]{article}

\usepackage[english]{babel}
\usepackage{lastpage}
\usepackage{amssymb,amsmath}
\usepackage{graphicx}
\usepackage{subfig}
\usepackage{caption}
\usepackage{euscript}
\usepackage[a4paper, portrait, margin=1in]{geometry}

\usepackage[bookmarks=true,hidelinks]{hyperref}
\usepackage{units}
\usepackage{bbm}

\usepackage{amsthm}

\newtheorem{theorem}{Theorem}
\newtheorem{proposition}[theorem]{Proposition}

\newtheorem{remark}[theorem]{Remark}

\newtheorem{lemma}[theorem]{Lemma}

\numberwithin{equation}{section}
\numberwithin{theorem}{section}

\newcommand{\proj}{\mathcal{A}}

\newcommand{\evo}{\mathcal{S}}
\newcommand{\R}{\mathbb{R}}
\newcommand{\N}{\mathbb{N}}
\newcommand{\Z}{\mathbb{Z}}
\newcommand{\Dx}{{\Delta x}}
\newcommand{\Dt}{{\Delta t}}

\renewcommand{\leq}{\leqslant}
\renewcommand{\geq}{\geqslant}
\renewcommand{\phi}{\varphi}
\newcommand{\Lip}{\mathrm{Lip}}
\newcommand{\DLip}{\mathrm{DLip}}

\newcommand{\hf}{{\unitfrac{1}{2}}}
\newcommand{\thf}{{\unitfrac{3}{2}}}
\newcommand{\iphf}{{i+\hf}}

\newcommand{\imhf}{{i-\hf}}
\newcommand{\imthf}{{i-\thf}}

\newcommand{\cell}{{\mathcal{C}}}
\renewcommand{\epsilon}{\varepsilon}

\usepackage[dvipsnames]{xcolor}
\definecolor{myred}{HTML}{FF3D3D}
\definecolor{mycyan}{HTML}{0474BE}
\definecolor{mygreen}{HTML}{1EB6D6}

\usepackage{tikz}
\usepackage{pgfplots}
\pgfplotsset{width=.4\textwidth,compat=newest}
\usepgfplotslibrary{fillbetween}

\usepackage{booktabs}

\allowdisplaybreaks

\title{The optimal convergence rate of monotone schemes for conservation laws in the Wasserstein distance}

\author{Adrian M. Ruf\thanks{Department of Mathematics, University of Oslo, Norway (adrianru@math.uio.no)} , Espen Sande\footnotemark[1]\,\,\thanks{Department of Mathematics, University of Rome Tor Vergata, Italy (sande@mat.uniroma2.it)}\; and Susanne Solem\thanks{Department of Mathematical Sciences, NTNU, Trondheim, Norway (susanne.solem@ntnu.no)\newline
A.\,M. Ruf has received funding from the European Union’s Framework Programme for Research and Innovation Horizon 2020 (2014-2020) under the Marie Sk{\l}odowska-Curie Grant Agreement No. 642768.
E. Sande was supported by the European Research Council under the European Union’s Seventh Framework Programme (FP7/2007-2013) / ERC grant agreement 339643.}}

\begin{document}
%
%
%
%
%
%

\maketitle

\begin{abstract}
In 1994, Nessyahu, Tadmor and Tassa studied convergence rates of monotone finite volume approximations of conservation laws. For compactly supported, $\Lip^+$-bounded initial data they showed a first-order convergence rate in the Wasserstein distance. Our main result is to prove that this rate is optimal.
We further provide numerical evidence indicating that the rate in the case of $\Lip^+$-unbounded initial data is worse than first-order.
\end{abstract}

\section{Introduction}
In their 1994 paper, Nessyahu, Tadmor and Tassa \cite{nessyconv} showed that a large class of monotone finite volume methods converge to the entropy solution of the hyperbolic conservation law
\begin{equation}\label{eq:conslaw}
\begin{aligned}
u_t +f(u)_x &= 0, \qquad x \in \R, \quad t >0,\\
u(x,0)&=u_0(x),
\end{aligned}
\end{equation}
at a rate of $O(\Dx)$ in the 1-Wasserstein distance $W_1$ (using the different name $\Lip'$) under the assumption that $f$ is strictly convex ($f'' \geq \alpha > 0$) and the initial datum $u_0$ is compactly supported and $\Lip^+$-bounded, i.e.
\begin{equation}\label{eq:OSLC}
\frac{u_0(x+z)-u_0(x)}{z} \leq C, \qquad \forall \ x,z \in \R, \quad z \neq 0.
\end{equation}
Recently, Fjordholm and Solem \cite{fjordholm2016second} showed a convergence rate of $O(\Dx^2)$ in $W_1$ for initial data consisting of finitely many shocks. This raises the question whether the first-order rate in $W_1$ of \cite{nessyconv} can be improved. In this paper we show that this is not possible. We construct a compactly supported and $\Lip^+$-bounded initial datum for which the convergence rate in $W_1$ is no better than first-order. In other words, the rate $O(\Dx)$ in \cite{nessyconv} is optimal.

\subsection{The Wasserstein distance}\label{sec: Wasserstein}
In one dimension, the $W_1$-distance between two functions $u$ and $v$ takes on the simple form
\begin{align*}
W_1(u,v) = \int_\R \left|\int_{-\infty}^x (u(y)-v(y)) \, dy\right| dx,
\end{align*}
see \cite{Villani} for more details.
In higher dimensions it is the dual of the $\Lip$-norm, referred to as $\Lip'$ in \cite{nessyahu1994convergence}.
This distance was first utilized in the context of conservation laws in a series of papers by Nessyahu, Tadmor and Tassa \cite{tad91, nessyconv92, nessyconv} where they, among other things, prove convergence rates for a large class of approximations to the solution of the conservation law \eqref{eq:conslaw} in the $W_1$-distance.

Heuristically, one can think of the $W_1$-distance as measuring the minimal amount of work needed to move mass from one place to another. In the case of increasing initial data, a monotone scheme provides an approximation of the type shown in Figure \ref{fig: heuristic graph} after some time has elapsed.
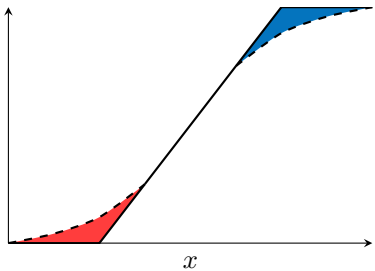
\begin{figure}[h]
\centering
\begin{tikzpicture}[yscale=.8]
\begin{axis}[axis lines=left, xtick = \empty, ytick = \empty,xlabel={$x$}]
\addplot+[sharp plot,mark=none, thick, black]
coordinates{
(0.375 , 0.25)
(0.625 , 0.75)
};
\addplot+[sharp plot, mark=none, thick, black, name path=A]
coordinates{
(0,0)
(0.25,0)
(0.375 , 0.25)
};
\addplot+[mark=none,thick, dashed, black, smooth, name path=B]
coordinates{%
(0 , 0)
(0.125 , 0.04)
(0.25 , 0.11)
(0.375 , 0.25)
};
\addplot+[myred] fill between[of =A and B];
\addplot+[sharp plot, mark=none, thick, black, name path=C]
coordinates{
(0.625 , 0.75)
(0.75,1)
(1,1)
};
\addplot+[mark=none,thick, dashed, black, smooth, name path=D]
coordinates{
(0.625 , 0.75)
(0.75  , 0.89)
(0.875 , 0.96)
(1 , 1)
};
\addplot+[mycyan] fill between[of =C and D];
\end{axis}
\end{tikzpicture}
\caption{Exact and approximate solution of \eqref{eq:conslaw}.}\label{fig: heuristic graph}
\end{figure}
Given that the $L^1$ error is $O(\Dx)$, the surplus of mass on the left (red area) needed to be moved is $O(\Dx)$ and it needs to be moved a distance of $O(1)$ to the shortage of mass on the right (blue area), see \cite{solem2018convergence} for a similar argument. Therefore, we expect the $W_1$-error to be no better than $O(\Dx)\cdot O(1)=O(\Dx)$ in this case. The goal of this paper is to make this heuristic argument rigorous.

\subsection{Convergence rates for monotone schemes in \texorpdfstring{$L^1$}{L1} and \texorpdfstring{$W_1$}{W1}}

The $L^1$ convergence rate for monotone schemes is restricted to $O(\Dx)$ as these schemes are at most first-order accurate (see Harten, Hyman, and Lax \cite{harten1976finite}).
However, the generic result on convergence rates of monotone schemes for the conservation law \eqref{eq:conslaw} is the $O(\Dx^\hf)$ rate in $L^1$, due to Kuznetsov \cite{kuznetsov76} which dates back to $1976$. By constructing a (pathological) initial datum, \c{S}abac showed in $1997$ that the $O(\Dx^\hf)$ rate for monotone methods is, in fact, optimal and cannot be improved without further assumptions on the class of initial data \cite{sabac} (see \cite{tang1995sharpness} for the linear advection equation). For that, \c{S}abac assumed strict convexity of the flux ($f''\geq \alpha >0$), which is the setting considered in the present paper.

Although the convergence rate $O(\Dx^\hf)$ is optimal, in some special cases higher convergence rates for monotone schemes have been shown. For example Harabetian \cite{harabetian1988rarefactions} proved that monotone schemes for centered rarefaction waves converge at a rate of $O(\Dx|\log \Dx|)$ in $L^1$, which is claimed to be optimal in \cite{wang19981}. Before that, Bakhvalov \cite{bakhvalov1961estimation} proved the same rate for an upwind scheme in a weaker norm related to $W_1$. Wang \cite{wang19981} showed that the rate $O(\Dx|\log \Dx|)$ in $L^1$ also appears close to the critical time of shock formation in certain special cases. Furthermore, Teng and Zhang \cite{zhang_teng} proved that monotone schemes converge at the optimal rate of $O(\Dx)$ in $L^1$ provided the initial datum is piecewise constant with a finite number of discontinuities that only allow for shocks at later times (in the case of a convex flux this means only downward jumps). Later, this result was extended to the convergence rate $O(\Dx^2)$ in $W_1$ by Fjordholm and Solem \cite{fjordholm2016second}.

The seminal work on convergence rates in the Wasserstein distance is Nessyahu and Tadmor's 1992 paper \cite{nessyconv92}. Using the dual equation studied by Tadmor in \cite{tad91}, the authors showed that conservative, $\Lip^+$-stable and $\Lip'$-consistent schemes converge at a rate of $O(\Dx)$ in the Wasserstein distance, for $\Lip^+$-bounded (i.e., rarefaction-free), compactly supported initial data. Examples of schemes that satisfy these assumptions are the Lax--Friedrichs, Engquist--Osher, and Godunov scheme.
Nessyahu, Tadmor and Tassa later used that framework to prove the same convergence rate for so-called Godunov type schemes. In addition to the aforementioned schemes, a subset of (formally) second-order MUSCL schemes also falls into this class.
Notably, Nessyahu and Tassa \cite{nessyahu1994convergence} also covered the case of $\Lip^+$-unbounded initial data and showed a convergence rate of $O(\epsilon|\log\epsilon|)$ in $W_1$ for viscous regularizations of \eqref{eq:conslaw}.

Table \ref{tab:survey of convergence rates} provides an overview of the results concerning convergence rates for monotone schemes in both $L^1$ and the Wasserstein distance.

\begin{table}[h]
\centering
\renewcommand{\arraystretch}{1.5}
\begin{tabular}{p{.4\textwidth}cccc}
	\toprule
	Case considered & $L^1$ rate & Optimal & $W_1$ rate & Optimal\\
	\midrule
	General $L^1\cap BV$ initial data \cite{kuznetsov76}& $\Dx^\hf$ & \cite{sabac} & --&--\\
	$\Lip^+$-bounded, compactly supported initial data \cite{kuznetsov76,nessyconv92,nessyconv} & $\Dx^\hf$ & --&$\Dx$ & Thm. \ref{th:optimality} \\
	Rarefaction solutions \cite{harabetian1988rarefactions} & $\Dx|\log\Dx|$ &\cite{harabetian1988rarefactions,wang19981} &--&-- \\
	Decreasing data (before shocks) \cite{wang19981} &$\Dx|\log\Dx|$ & \cite{wang19981} &-- &--\\
	Decreasing piecewise constant initial datum (finitely many pieces) \cite{zhang_teng,fjordholm2016second} & $\Dx$ &\cite{zhang_teng} & $\Dx^2$ & Rem. \ref{rem:ineq}\\
	\bottomrule
\end{tabular}
\caption{Short overview of results regarding rates of convergence for monotone schemes.}\label{tab:survey of convergence rates}
\end{table}

\begin{remark}\label{rem:ineq}
As remarked in \cite{nessyconv92} and \cite{fjordholm2016second} one can recover the well-known half-order rate in $L^1$ from the first-order rate in $W_1$ by utilizing the Sobolev interpolation inequality 
$\|Dg\|_{L^1(\R)}\leq C \|D^2g\|_{L^1(\R)}^\hf\|g\|_{L^1(\R)}^\hf,$
see e.g. \cite[Thm.~9.3]{Fried}, as follows.
Let $u$ be the solution of the conservation law \eqref{eq:conslaw} and $u_\Dx$ a monotone approximation to it. Then let $Dg$ be a suitable approximation of the error, $u-u_\Dx$, such that $\|D^2 g\|_{L^1}$ is bounded by the total variation of $u_0$. Then it follows that $\|u-u_\Dx\|_{L^1(\R)}\leq C \operatorname{TV}(u_0)^\hf W_1(u,u_\Dx)^\hf$. Note that this inequality can also be used in the other direction: The optimality of the convergence rate $O(\Dx)$ in $L^1$ for piecewise constant, decreasing initial data implies the optimality of the convergence rate $O(\Dx^2)$ in $W_1$.
Moreover, optimality of the convergence rate $O(\Dx^\hf)$ in $L^1$ for general $L^1\cap BV$ initial data implies that the convergence rate in $W_1$ cannot be better than $O(\Dx)$ in the general case.
\end{remark}

\section{Preliminaries}
In this section we present the class of numerical methods and the observations needed to prove the optimal rate.
\subsection{Monotone schemes and first-order convergence in \texorpdfstring{$W_1$}{W1}}
Let $x_\imhf$, $i \in \Z$, be equidistant points, $\Dx$ apart, and let $\cell_i = [x_\imhf, x_\iphf)$, $t^n = n\Dt$, $n\in\N$, and $\lambda = \Dt/ \Dx$. 
Then we consider schemes of the form
\begin{equation}\label{eq:monotonemethods}
\begin{split}
u_i^{n+1} &= u_i^{n}-\lambda \left(F(u_i^{n},u_{i+1}^{n}) - F(u_{i-1}^{n},u_{i}^{n})\right), \\
u_i^0 &= \frac{1}{\Dx} \int_{\cell_i} u_0(x)\, dx,
\end{split}
\qquad\qquad (i\in\Z)
\end{equation}
where the numerical flux function $F(\cdot,\cdot)$ is consistent with the flux $f$, i.e. $F(u,u)=f(u)$. The scheme is monotone if and only if $F(\cdot,\cdot)$ is nondecreasing in the first argument and nonincreasing in the second and a certain CFL condition is satisfied.

A numerical approximation $u_\Dx$, defined by $u_\Dx(x,t)=u_i^{n}$ for $(x,t)\in\cell_i\times[t^n,t^{n+1})$, is said to be discrete $\Lip^+$-stable if
\begin{align*}
\|u_\Dx(t)\|_{\DLip^+} := \max_x \frac{u_\Dx(x+\Dx,t) - u_\Dx(x,t)}{\Dx} \leq C, \quad \forall t>0,
\end{align*}
and a numerical method is called $W_1$-consistent in \cite{nessyconv92,nessyconv} if $$W_1(u_{\Dx}(\cdot,0),u_0)+W_1((u_{\Dx})_t,-f(u_{\Dx})_x)\leq O(\Dx),$$
where the second term is the Wasserstein distance for functions in space-time.
Here, the derivatives involving $u_{\Dx}$ are defined in \cite{nessyconv}.
Let $\evo_t$ be the exact evolution operator to \eqref{eq:conslaw}, i.e. such that $\evo_t u(\cdot,s)=u(\cdot,s+t)$, and let $\proj$ be the piecewise constant projection operator, 
\begin{align}
\proj u (x) = \frac{1}{\Dx} \int_{\cell_i} u(y)\, dy, \quad x \in \cell_i. \label{projection operator}
\end{align}
Then the method \eqref{eq:monotonemethods} is $W_1$-consistent if it can be rewritten in a form $u_i^{n+1}=P\proj \evo u_\Dx(t^{n}+)$. Here $u_\Dx(t^{n}+)$ is the numerical approximation calculated with \eqref{eq:monotonemethods} at $t^n$ and $P$ is a scheme-dependent projection operator that satisfies
\begin{align}\label{eq:w1consistent}
W_1(P \proj v, v) \leq O(\Dx^2)\operatorname{TV}(v),
\end{align}
see \cite[Thm. 2.1]{nessyconv}.
Assuming that the numerical approximation is $\Lip^+$-stable and $W_1$-consistent, Nessyahu et al. proved the following.
\begin{theorem}[Nessyahu et al. {\cite[Thm.~2.3]{nessyconv}}]\label{th:nessy}
Assume that $u_0$ is compactly supported and $\Lip^+$-bounded, see \eqref{eq:OSLC}, and that the numerical approximation $u_\Dx$ is discrete $\Lip^+$-stable and $W_1$-consistent. Then for any $T>0$,
\begin{align*}
W_1\big(u(t),u_\Dx(t)\big) \leq C_T \Dx,
\end{align*}
for all $0\leq t\leq T$.
\end{theorem}
Note that this theorem includes all monotone schemes and even some (formally) higher-order schemes, as long as they are discrete $\Lip^+$-stable and satisfy \eqref{eq:w1consistent}. Examples of monotone numerical schemes that satisfy these assumptions are the Godunov, Lax--Friedrichs and Enquist--Osher schemes, see \cite{nessyconv}. 

The Godunov scheme is one (or \emph{the}) example of a monotone scheme. It consists of piecewise constant projection and exact evolution in time. Using the operators $\evo$ and $\proj$, it can be written in the simple form 
\begin{align}\label{eq:numericalsolution}
u_\Dx(x,t) = \left(\evo_{t-t^n}\proj\left(\evo_{\Dt}\proj \right)^n u_0\right)(x,t), \qquad t \in [t^n,t^{n+1}).
\end{align}

\subsection{The error equation}\label{subsec: error equation}
Let $E(x,t) = \int_{-\infty}^x (u(y,t)-v(y,t)) \, dy$, where $u$ and $v$ are two solutions to \eqref{eq:conslaw}, possibly with different initial data $u_0$ and $v_0$. Then
\begin{equation*}
W_1(u(t),v(t))=\|E(t)\|_{L^1(\R)}.
\end{equation*}
Due to the $L^1$-contraction property of solutions to \eqref{eq:conslaw}, $E$ is Lipschitz continuous in both time and space. 
Now let
\begin{equation*}
	a(u,v) = \int_0^1 f'(\gamma u +(1-\gamma)v) \, d\gamma.
\end{equation*}
Then, as $u$ and $v$ are solutions of \eqref{eq:conslaw},
$E$ satisfies the transport equation
\begin{equation}\label{eq:erroreq}
\begin{split}
E_t +a(u,v)E_x & = 0,  \\
E(x,0) & = \int_{-\infty}^x(u(y,0)-v(y,0)) \, dy.
\end{split}
\end{equation}
Hence, if $E(0)\geq 0$, then $E(t) \geq 0$ at any later time $t>0$. Given nondecreasing initial data and a strictly convex flux, $u$ and $v$ are continuous for $t>0$ and thus $a(u,v)$ is as well. It follows that \eqref{eq:erroreq} is well-defined for any $t>0$.

After an integration by parts, we can see that the time-derivative of the Wasserstein distance between two solutions of \eqref{eq:conslaw} satisfies
\begin{align*}
\frac{d}{dt} \|E(t)\|_{L^1(\R)} \ = \ -\int_\R a(u,v)(t)\partial_x|E(t)| \, dx \ = \ \int_\R D_x a(u,v)(t) |E(t)| \,dx,
\end{align*}
where $D_x a(u,v)$ is to be understood as the distributional derivative of $a(u,v)$. 
Note that since the flux is strictly convex ($f'' \geq \alpha > 0$) we have
\begin{equation*}
	D_x a(u,v) \geq \frac{\alpha}{2}(u_x+v_x).
\end{equation*}
It follows that $D_x a(u,v)$, and consequently $\frac{d}{dt}\|E(t)\|_{L^1(\R)}$, is nonnegative, if $u$ and $v$ are nondecreasing.


\subsection{The projections}
This section contains two useful lemmas on the projection operator $\proj$ in the case of nondecreasing functions. The first one shows that the primitive of the projection error is nonnegative.
\begin{lemma}\label{lem: positivity projection}
For a nondecreasing function $v$, the projection operator $\proj$ defined in \eqref{projection operator} satisfies
\begin{align}\label{eq:positive_proj}
\int_{-\infty}^x \big(\proj v - v\big) (y) \, dy \geq 0
\end{align}
for all $x\in\R$.
\end{lemma}
\begin{proof}
For $x\in\cell_i$ we find that
\begin{align*}
	\int_{-\infty}^x (\proj v - v)(y)\, dy  = \int_{-\infty}^{x_\imhf}(\proj v -v)(y)\, dy + \int_{x_{\imhf}}^x(\proj v -v)(y)\, dy,
\end{align*}
where the first term vanishes due to conservation of mass of $\proj$. As $\proj v $ is constant in $\cell_i$ and is the average of the function $v$ which is nondecreasing,
\begin{align*}
\int_{x_{\imhf}}^x(\proj v -v)(y)\, dy \geq 0,
\end{align*}
and we can conclude that \eqref{eq:positive_proj} holds.
\end{proof}
The second lemma states that the projection operator $\proj$ preserves positivity of the difference between the primitives.
\begin{lemma}\label{lem:positivity} Let $u$ and $v$ be two nondecreasing functions. Then
\begin{align}\label{eq:positivity}
\int_{-\infty}^x (v-u)(y) \, dy  \geq 0 \quad \forall \ x\in\R \quad \Rightarrow \quad \int_{-\infty}^x (\proj v- u)(y) \, dy  \geq 0 \quad \forall \ x\in\R.
\end{align}
\end{lemma}
\begin{proof}
Assume that the inequality to the left in \eqref{eq:positivity} holds. Then, using Lemma \ref{lem: positivity projection},
\begin{align*}
\int_{-\infty}^x (\proj v-u)(y) \, dy =& \int_{-\infty}^{x} ( v-u)(y) \, dy + \int_{-\infty}^x \big(\proj v - v\big) (y) \, dy \ \geq 0.
\end{align*}
\end{proof}

\section{Optimality}
With the observations in the previous section, we can prove that the first-order rate in Theorem~\ref{th:nessy} is optimal in $W_1$:
\begin{theorem}\label{th:optimality}
Let $f$ be strictly convex and let $u_0$ be compactly supported and $\Lip^+$-bounded, i.e., satisfy \eqref{eq:OSLC}. Then the optimal convergence rate in the Wasserstein distance of monotone finite volume schemes \eqref{eq:monotonemethods} satisfying \eqref{eq:w1consistent}, is $O(\Dx)$. 
\end{theorem}
\noindent We postpone the (short) proof to the end of the section.

The initial datum $u_0$ has to be $\Lip^+$-bounded and compactly supported for Theorem~\ref{th:nessy} to hold. We will therefore consider compactly supported initial data $u_0$ consisting of one increasing, $\Lip^+$-bounded part, increasing from $0$ to $M$ and one decreasing part, decreasing from $M$  to $0$.
One realization of a suitable initial datum is
\begin{align}\label{eq:initialdata}
u_0(x)=
\begin{cases}
0, \quad & \phantom{x_M\leq ~\,} x < x_s,\\
M\frac{x-x_s}{x_0-x_s}, \quad  & \mathop{\rlap{$x_s$}}\phantom{x_M}\leq x <x_0, \\
M, \quad & \mathop{\rlap{$x_0$}}\phantom{x_M} \leq x <x_M,\\
M \left(1-\frac{x-x_M}{x_e-x_M}\right), \quad & \mathop{\rlap{$x_M$}}\phantom{x_M} \leq  x <x_e, \\
0, \quad & \phantom{x_M \leq ~\,} x\geq x_e,
\end{cases}
\end{align}
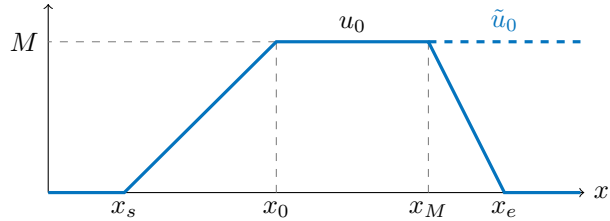
\begin{figure}[h]
\begin{center}
\begin{tikzpicture}	
	\draw[->] (0,0) -- (7.05,0) node[at end, right] {$x$};
	\draw[->] (0,0) -- (0,2.5);
	\draw[dashed, gray] (3,2) -- (0,2) node[left, black] {$M$};
	\node[below] at (1,0) {$x_s$};
	\draw[dashed, gray] (3,2) -- (3,0) node[below, black] {$x_0$};
	\draw[dashed, gray] (5,2) -- (5,0) node[below, black] {$x_M$};
	\node[below] at (6,0) {$x_e$};
	\draw[mycyan,very thick] (0,0) -- (1,0) -- (3,2) -- (5,2) -- (6,0) -- (7,0);
	\node[above] at (4,2) {$u_0$};
	\draw[mycyan,very thick, dashed] (5,2) -- (7,2) node[midway,above]{$\tilde{u}_0$};
\end{tikzpicture}
\caption{The initial datum used to show optimality of the convergence rate}
\label{fig: initial datum}
\end{center}
\end{figure}
where $x_s<x_0<x_M<x_e$, and $[x_s,x_e]$ is the support of $u_0$, see Figure \ref{fig: initial datum}.

Since we only expect the increasing part of $u_0$ to contribute to the reduced convergence rate $O(\Dx)$, we will simplify our calculations by seperating the increasing parts of $u$ and $u_\Dx$ from the decreasing parts.
To that end, for any fixed time $T>0$, without restrictions, we will assume that $x_M-x_0$ is big enough such that there exists an $x^*$ satisfying $x_0<x^*<x_M$ where $u_\Dx(x^*+f'(M)t,t) = M$ for all $0\leq t\leq T$. Using $x^*$ we introduce the increasing auxiliary functions
\begin{align}
\begin{split}
	\tilde{u}_0(x) &= \begin{cases}
		u_0(x),  &  x<x^* \\
		M, \quad & x \geq x^*,
	\end{cases}\\
	\tilde{u}(x,t) &= \begin{cases}
		u(x,t), & x<x^* +f'(M)t,\\
		M, & x\geq x^* +f'(M)t,
	\end{cases}\\
	\tilde{u}_\Dx (x,t) &= \begin{cases}
		u_\Dx(x,t), & x<x^* +f'(M)t\\
		M, & x\geq x^* +f'(M)t.
	\end{cases}
\end{split}\label{definition tilde}
\end{align}
Then, the assumption above implies
\begin{equation*}
	\int_{-\infty}^{x^*+f'(M) t} \left(\tilde{u}-\tilde{u}_\Dx\right)\, dy = \int_{-\infty}^{x^*+f'(M) t} \left(u-u_\Dx\right)\, dy = 0,
\end{equation*}
and therefore
\begin{equation*}
	W_1(u(t),u_\Dx(t)) \geq W_1\left(\tilde{u}(t),\tilde{u}_\Dx (t)\right).
\end{equation*}
We now estimate the $W_1$-error between $\tilde{u}$ and $\tilde{u}_\Dx$ from below.
\begin{proposition}\label{prop:biggerthansum}
Let $u_0$ be as described above and $u_\Dx$ the numerical approximation \eqref{eq:numericalsolution}. Then for $0<t<T$, 
\begin{align*}
W_1 \left(u(t), u_\Dx (t) \right) \geq \sum_{n=0}^{N} W_1\left(\proj \tilde{u}_\Dx(t^n-), \tilde{u}_\Dx(t^n-) \right),  
\end{align*}
where $N$ is such that $t \in [t^N, t^{N+1})$ and $u_\Dx(t^n-)$ is the numerical approximation right before averaging.
\end{proposition}
\begin{proof}
Let $t \in [t^N, t^{N+1})$ and $E_\Dx(x,t) = \int_{-\infty}^x (\tilde{u}_\Dx - \tilde{u})(y,t) \, dy$. Because of Lemma \ref{lem: positivity projection}
\begin{equation*}
	E_{\Dx}(x,0)=\int_{-\infty}^x (\tilde{u}_{\Dx}- \tilde{u})(y,0)\, dy = \int_{-\infty}^x (\proj \tilde{u}_0-\tilde{u}_0)(y)\,dy \geq 0
\end{equation*}
for all $x\in\R$.
The fact that $E_\Dx$ satisfies the transport equation \eqref{eq:erroreq} and Lemma \ref{lem:positivity} imply that $E_\Dx$ is nonnegative for all $t \geq 0$. Hence, 
\begin{align*}
W_1 \left(u(t), u_\Dx (t) \right)\geq W_1\left(\tilde{u}(t),\tilde{u}_\Dx(t)\right) = \int_\R E_{\Dx}(x,t)\, dx= \int_\R  \int_{-\infty}^x (\tilde{u}_\Dx-\tilde{u})(y,t) \, dy dx.
\end{align*}
As $\tilde{u}_0$ is nondecreasing and the conservation law \eqref{eq:conslaw} and the scheme \eqref{eq:monotonemethods} are monotonicity preserving, $\tilde{u}$ and $\tilde{u}_\Dx$ will be nondecreasing at any later time. It follows from the argument in Section \ref{subsec: error equation} that for $t\in [t^N,t^{N+1})$ the $W_1$ error between $\tilde{u}$ and $\tilde{u}_\Dx$ will be nondecreasing. Hence,
\begin{align*}
W_1 \left(\tilde{u}(t), \tilde{u}_\Dx (t) \right) & \geq W_1 \left(\tilde{u}(t^N), \proj \tilde{u}_\Dx (t^N) \right) \\
& = \int_\R  \int_{-\infty}^x (\proj \tilde{u}_\Dx-\tilde{u})(y,t^N) \, dy dx \\
 & = \int_\R  \int_{-\infty}^x (\proj \tilde{u}_\Dx - \tilde{u}_\Dx)(y,t^N-) \, dy dx + \int_\R  \int_{-\infty}^x (\tilde{u}_\Dx-\tilde{u})(y,t^N-) \, dy dx,
\end{align*}
where, in the last line, we have added and subtracted $\tilde{u}_\Dx(t^N)$. We can now continue the same procedure on the last term in the above $N$ times, and we end up with
\begin{align*}
W_1 \left(u(t), u_\Dx (t) \right) \geq \sum_{n=0}^{N} \int_\R  \int_{-\infty}^x (\proj \tilde{u}_\Dx - \tilde{u}_\Dx)(y,t^n-) \, dy dx = \sum_{n=0}^{N} W_1 \left(\proj \tilde{u}_\Dx(t^n-), \tilde{u}_\Dx (t^n-) \right),
\end{align*}
which is what we wanted to prove.
\end{proof}

In order to conclude that the $O(\Dx)$ rate is optimal in $W_1$, we need to show that for the increasing part of $u_0$ the projection error $W_1 \left(\proj \tilde{u}_\Dx(t^n-), \tilde{u}_\Dx (t^n-) \right)$ is bounded from below by $C \Dx\Dt$ for any $0\leq t^n < T$. 
\begin{proposition}\label{prop:W1projerror}
Let $u_0$ be increasing and assume that $\beta_1 \geq f' \geq \beta_2 > 0$ on $[-M,M]$, where $M = \|u \|_{L^\infty(\R)}$. Then if 
\begin{align*}
\lambda \leq \frac{1}{2\beta_1},
\end{align*} 
we have
\begin{align*}
W_1\left(\proj u_\Dx(t^n-), u_\Dx(t^n-) \right) \geq \frac{\Dx\Dt}{2}\left(1-\beta_1 \lambda \right)\beta_2  \mathrm{TV} \left( u_0 \right).
\end{align*}
\end{proposition}
\begin{proof}
From the positivity of the projection error and the conservation of mass in each cell $\cell_i$, we get that
\begin{align*}
W_1\left(\proj u_\Dx(t^n-), u_\Dx(t^n-) \right) & = \int_\R \int_{-\infty}^x (\proj u_\Dx - u_\Dx)(y,t^n-) \, dydx \\
&= \sum_i \int_{\cell_i}\int_{x_\imhf}^x \left( u_i^n - u_\Dx(y,t^n-) \right) \, dydx.
\end{align*}
Each term in the sum can be rewritten in the following way,
\begin{align}
\int_{\cell_i}\int_{x_\imhf}^x & \left( u_i^n - u_\Dx(y,t^n-) \right) \, dydx \notag \\
& =  \ \int_{\cell_i} (x-x_\imhf) \left( u_i^n - \frac{1}{(x-x_\imhf)}\int_{x_\imhf}^x u_\Dx(y,t^n-) \, dy \right) dx \notag\\
&=  \ \int_{\cell_i} \frac{1}{2}(x-x_\imhf)^2 \left( \frac{u_\Dx(x,t^n-)}{x-x_\imhf} - \frac{1}{(x-x_\imhf)^2}\int_{x_\imhf}^x u_\Dx(y,t^n-) \, dy \right) dx \notag\\
&=  \ \frac{1}{2} \int_{\cell_i} \left(u_\Dx(x,t^n-) (x-x_\imhf) - \int_{x_\imhf}^x u_\Dx(y,t^n-) \, dy \right) dx \notag\\
&=  \ \frac{1}{2} \int_{\cell_i} \int_{x_\imhf}^x ( u_\Dx (x,t^n-) -u_\Dx(y,t^n-) ) \, dydx, \label{last term}
\end{align}
after an integration by parts in the variable $x$ from the first to the second line (which is justified by the fact that $u_\Dx$ is continuous for $t>0$ as $f' \geq \beta_2 > 0$ and $f'' \geq \alpha > 0$).
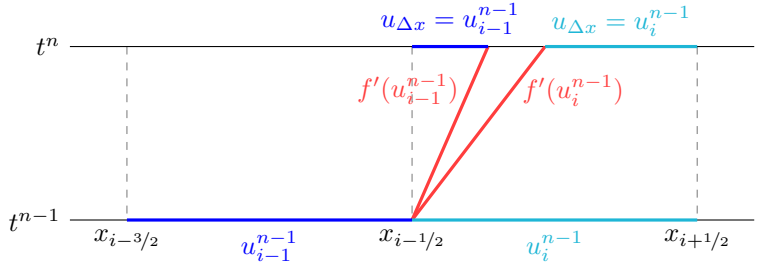
\begin{figure}[h]
\begin{center}
\begin{tikzpicture}[xscale=2.5, yscale=2.3]
	\draw[gray,dashed] (.5,1) -- (.5,0) node[below, black] {$x_\imthf$};
	\draw[gray,dashed] (2,1) -- (2,0) node[below, black] {$x_\imhf$};
	\draw[gray,dashed] (3.5,1) -- (3.5,0) node[below, black] {$x_\iphf$};
	\draw (3.8,0) -- (0.2,0) node[left] {$t^{n-1}$};
	\draw (3.8,1) -- (0.2,1) node[left] {$t^n$};
	\draw[myred,very thick] (2,0) -- (2.4,1) node[near end,left] {$f'(u_{i-1}^{n-1})$};
	\draw[blue,very thick] (2,1) -- (2.4,1) node[midway,above] {$u_\Dx=u_{i-1}^{n-1}$};
	\draw[myred,very thick] (2,0) -- (2.7,1) node[near end,right] {$f'(u_i^{n-1})$};
	\draw[mygreen,very thick] (2.7,1) -- (3.5,1) node[midway,above] {$u_\Dx =u_i^{n-1}$};
	\draw[blue,very thick] (.5,0) -- (2,0) node[midway,below] {$u_{i-1}^{n-1}$};
	\draw[mygreen,very thick] (2,0) -- (3.5,0) node[midway,below] {$u_i^{n-1}$};
\end{tikzpicture}
\end{center}
\caption{Transportation of the cell averages $u_{i-1}^{n-1}$ and $u_i^{n-1}$ when calculating $u_\Dx$ in $\cell_i$.}\label{fig: Riemann problem in cell_i}
\end{figure}
By integrating only over the part of $\cell_i$ where $u_\Dx$ is constant and ignoring the Riemann fan between $x_{\imhf}+f'(u_{i-1}^{n-1})\Dt$ and $x_{\imhf}+f'(u_i^{n-1})\Dt$ (see Figure \ref{fig: Riemann problem in cell_i}) we can bound the last term \eqref{last term} from below as follows:

\begin{align*}
\ \frac{1}{2} \int_{\cell_i} \int_{x_\imhf}^x & ( u_\Dx (x,t^n-) -u_\Dx(y,t^n-) ) \, dydx \\
& \geq   \frac{1}{2} \int_{x_\imhf+f'(u_i^{n-1})\Dt}^{x_\iphf} \int_{x_\imhf}^x ( u_\Dx (x,t^n-) -u_\Dx(y,t^n-) ) \, dydx \\
& \geq  \frac{1}{2} \int_{x_\imhf+f'(u_i^{n-1})\Dt}^{x_\iphf} \int_{x_\imhf}^{x_\imhf+f'(u_{i-1}^{n-1})\Dt} \left(u_\Dx (x,t^n-) -u_\Dx(y,t^n-)\right) \, dydx \\
& =  \frac{1}{2} f'(u_{i-1}^{n-1})\Dt \int_{x_\imhf+f'(u_i^{n-1})\Dt}^{x_\iphf} \left( u_\Dx(x,t^n-)- u_{i-1}^{n-1} \right) \, dx\\
& =  \frac{1}{2} \left(\Dx - f'(u_i^{n-1})\Dt \right)f'(u_{i-1}^{n-1})\Dt \left(u_i^{n-1}-u_{i-1}^{n-1} \right)
\end{align*}
Then, summing up,
\begin{align*}
\int_{\cell_i}\int_{x_\imhf}^x \left( u_\Dx(x,t^n-) - u_\Dx(y,t^n-) \right) \, dydx \geq & \  \frac{\Dx\Dt}{2}(1-f'(u_i^{n-1})\lambda)f'(u_{i-1}^{n-1}) \left(u_i^{n-1}-u_{i-1}^{n-1} \right) \\
\geq & \ \frac{\Dx\Dt}{2}(1-\beta_1\lambda)\beta_2 \left(u_i^{n-1}-u_{i-1}^{n-1} \right).
\end{align*}
Summing over all $i$, the result follows.
\end{proof}
We are now ready to prove Theorem \ref{th:optimality}.
\begin{proof}[Proof of Theorem \ref{th:optimality}]
Let $f$ be strictly convex. Without loss of generality we can assume $\beta_1\geq f'\geq \beta_2>0$ on $[-M,M]$ (otherwise we consider $\hat{u}=u+C$ and $\hat{u}_\Dx = u_\Dx+C$ for some suitable constant $C$, which will not affect the Wasserstein distance).
Combining Proposition \ref{prop:W1projerror} with Proposition \ref{prop:biggerthansum}, we find that
\begin{align*}
W_1(u(t),u_\Dx(t)) \geq \frac{\beta_2}{2}(1-\beta_1\lambda)t^N \textrm{TV}(u_0) \Dx,
\end{align*}
for $t\in [t^N,t^{N+1})$, which concludes the proof of Theorem \ref{th:optimality}.
\end{proof}

\section{Numerical experiments}

To illustrate our result, we consider two numerical experiments using Burgers' equation,
\begin{equation*}
	u_t + \left(\frac{u^2}{2}\right)_x = 0,
\end{equation*}
on the interval $[-1,1]$ with the initial data
\begin{equation*}
	u^1_0(x) = \begin{cases}
		0, & \phantom{-0.75\leq}~ x<-0.75,\\
		2x+1.5 & -0.75\leq x <-0.25,\\
		1, & -0.25\leq x < \phantom{-}0.25,\\
		-4x+2 & \phantom{-}0.25\leq x < \phantom{-}0.5,\\
		0, & \phantom{-0.75\leq}~ x\geq \phantom{-}0.5,
	\end{cases}
	\qquad\text{and}\qquad
	u^2_0(x) = \begin{cases}
		0, & x<0,\\
		1, & x\geq 0.
	\end{cases}
\end{equation*}
The first initial datum, $u_0^1$, is an example of the compactly supported $\Lip^+$-bounded $u_0$ in \eqref{eq:initialdata}, and therefore fits into the context considered in this paper. The second initial datum, $u_0^2$, on the other hand is $\Lip^+$-unbounded.
For both experiments we use the Godunov scheme, i.e., the monotone scheme \eqref{eq:monotonemethods} with the numerical flux function
\begin{equation*}
	F(a,b) = \frac{1}{2}\max(\max(a,0)^2,\min(b,0)^2),
\end{equation*}
and $\frac{\Dt}{\Dx} = 0.5$.
The exact solution for Experiment $1$ (for $t<0.25$) and Experiment $2$ is
\begin{equation*}
	u^1(x,t) = \begin{cases}
		0, & \phantom{-0.75+t\leq}~ x<-0.75,\\
		5x/3+ 5/4 & -0.75\phantom{+t}~\leq x <-0.25+t,\\
		1, & -0.25+t\leq x <\phantom{-}0.25+t,\\
		-20x/3+ 10/3 & \phantom{-}0.25 +t\leq x < \phantom{-}0.5,\\
		0, & \phantom{-0.75+t\leq}~ x\geq \phantom{-}0.5,
	\end{cases}
	\qquad\text{and}\qquad
	u^2(x,t) = \begin{cases}
		0, & \phantom{0<}~x<0,\\
		x/t, & 0\leq x<t,\\
		1, & \phantom{0<}~ x\geq t,
	\end{cases}
\end{equation*}
respectively.
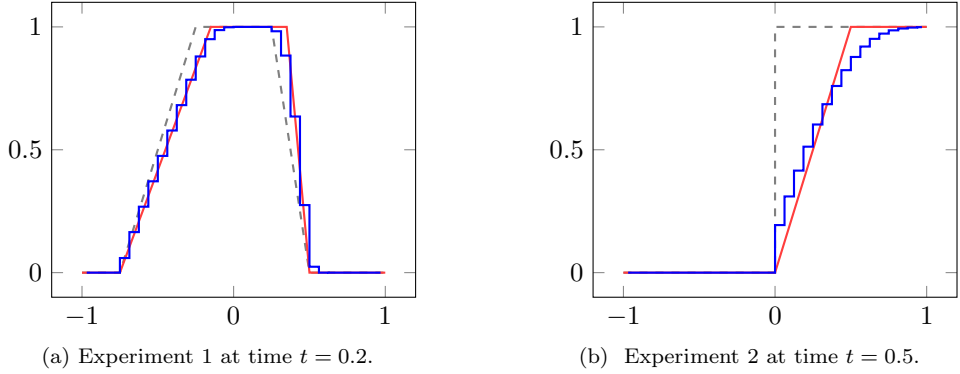
\begin{figure}[h]
\centering
\subfloat[Experiment $1$ at time $t=0.2$.]{
\begin{tikzpicture}
\begin{axis}
\addplot+[sharp plot, gray, mark=none, dashed, thick]
coordinates{
(-1,0)
(-0.75,0)
(-0.25,1)
(0.25,1)
(0.5,0)
(1,0)
};
\addplot+[sharp plot, myred, mark=none, thick]
coordinates{
(-1,0)
(-0.75,0)
(-0.15,1)
(0.35,1)
(0.5,0)
(1,0)
};
\addplot+[const plot mark mid,mark=none,blue,thick]
coordinates{%
( -0.96875 , 0.0 )
( -0.90625 , 0.0 )
( -0.84375 , 0.0 )
( -0.78125 , 0.0 )
( -0.71875 , 0.0594886208403 )
( -0.65625 , 0.164971916794 )
( -0.59375 , 0.268388873978 )
( -0.53125 , 0.371720484304 )
( -0.46875 , 0.475047619836 )
( -0.40625 , 0.578374603857 )
( -0.34375 , 0.681701587879 )
( -0.28125 , 0.785028571901 )
( -0.21875 , 0.879719199234 )
( -0.15625 , 0.950066897397 )
( -0.09375 , 0.987012624512 )
( -0.03125 , 0.998478999467 )
( 0.03125 , 1.0 )
( 0.09375 , 1.0 )
( 0.15625 , 1.0 )
( 0.21875 , 1.0 )
( 0.28125 , 0.982236030078 )
( 0.34375 , 0.883337578696 )
( 0.40625 , 0.635348236143 )
( 0.46875 , 0.274959989838 )
( 0.53125 , 0.0240642110838 )
( 0.59375 , 5.39541242233e-05 )
( 0.65625 , 3.76206552906e-11 )
( 0.71875 , 1.16415321827e-25 )
( 0.78125 , 0.0 )
( 0.84375 , 0.0 )
( 0.90625 , 0.0 )
( 0.96875 , 0.0 )
};
\end{axis}
\end{tikzpicture}
\label{fig: Graph bump}
}%
\hspace{4em}%
\subfloat[ Experiment $2$ at time $t=0.5$.]{
\begin{tikzpicture}
\begin{axis}
\addplot+[sharp plot, gray, mark=none, dashed, thick]
coordinates{
(-1,0)
(0,0)
(0,1)
(1,1)
};
\addplot+[sharp plot, myred, mark=none, thick]
coordinates{
(-1,0)
(0,0)
(0.5,1)
(1,1)
};
\addplot+[const plot mark mid,mark=none,blue,thick]
coordinates{
( -0.96875 , 0.0           )
( -0.90625 , 0.0           )
( -0.84375 , 0.0           )
( -0.78125 , 0.0           )
( -0.71875 , 0.0           )
( -0.65625 , 0.0           )
( -0.59375 , 0.0           )
( -0.53125 , 0.0           )
( -0.46875 , 0.0           )
( -0.40625 , 0.0           )
( -0.34375 , 0.0           )
( -0.28125 , 0.0           )
( -0.21875 , 0.0           )
( -0.15625 , 0.0           )
( -0.09375 , 0.0           )
( -0.03125 , 0.0           )
( 0.03125  , 0.193635350086 )
( 0.09375  , 0.310144302386 )
( 0.15625  , 0.415217703637 )
( 0.21875  , 0.51259219259  )
( 0.28125  , 0.602826525434 )
( 0.34375  , 0.685456449821 )
( 0.40625  , 0.759538537679 )
( 0.46875  , 0.823941125303 )
( 0.53125  , 0.877636937297 )
( 0.59375  , 0.920045029165 )
( 0.65625  , 0.95135557408  )
( 0.71875  , 0.972688598099 )
( 0.78125  , 0.985951792067 )
( 0.84375  , 0.993413486955 )
( 0.90625  , 0.997193230771 )
( 0.96875  , 0.998914031241 )
};
\end{axis}
\end{tikzpicture}
\label{fig: Graph upstep}
}%
\caption{Exact solution and numerical approximation and initial datum.}
\end{figure}
Figures \ref{fig: Graph bump} and \ref{fig: Graph upstep} show the initial data for Experiment 1 and 2 (respectively) in gray (dashed), the exact solutions in red (straight), and the numerical approximations calculated with the Godunov scheme in blue (piecewise constant).
\begin{table}[h]
\centering
\subfloat[Experiment 1 at time $t=0.2$.]{
\begin{tabular}{rcccc}
  \toprule
  \multicolumn{1}{c}{$n$} & $L^1$ OOC & $W_1$ OOC\\
  \midrule
	32  &  $0.822$ & $1.196$ \\
	64  &  $0.896$ & $1.123$ \\
	128 &  $0.861$ & $1.075$ \\
	256 &  $0.884$ & $1.046$ \\
	512 &  $0.900$ & $1.029 $ \\
  \bottomrule
\end{tabular}
\label{tbl: Rates bump}
}%
\hspace{4em}%
\subfloat[Experiment 2 at time $t=0.5$.]{
\centering
\begin{tabular}{rcccc}
  \toprule
  \multicolumn{1}{c}{$n$} & $L^1$ OOC & $W_1$ OOC\\
  \midrule
	32  &  $0.598$ & $0.764$ \\
	64  &  $0.641$ & $0.759$ \\
	128 &  $0.675$ & $0.761$ \\
	256 &  $0.708$ & $0.769$ \\
	512 &  $0.739$ & $0.782$ \\
  \bottomrule
\end{tabular}
\label{tbl: Rates upstep}
}%
\caption{Observed order of convergence in $L^1$ and $W_1$.}
\end{table}
%
%
Tables \ref{tbl: Rates bump} and \ref{tbl: Rates upstep} show the observed convergence rates of Experiment $1$ and $2$, where $n$ is the number of cells in the discretization. The first table clearly shows that the $W_1$ error is $O(\Dx)$ in Experiment $1$ and therefore numerically illustrates the optimality result of the present paper. The second table indicates that in the case of a single upward jump, i.e., $\Lip^+$-unbounded initial datum, we can expect a convergence rate of $O(\Dx|\log\Dx|)$ not only in $L^1$ as shown by Harabetian \cite{harabetian1988rarefactions}, but also in $W_1$ (see also Figure \ref{fig: OOC upstep}).
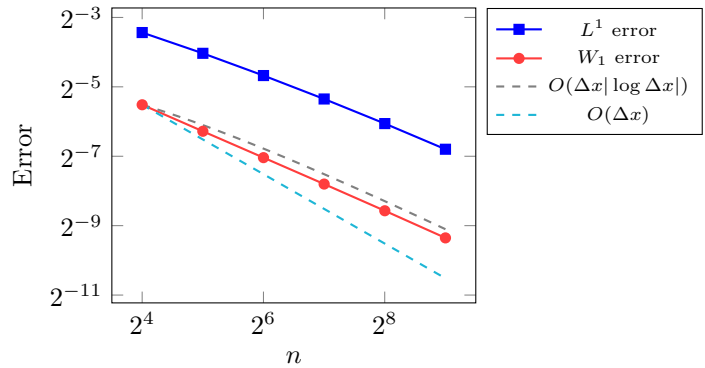
\begin{figure}[h]
\centering
\begin{tikzpicture}
\begin{loglogaxis}[
    log basis x={2},
    xlabel={$n$},
    log basis y={2},
    ylabel={Error},
    legend pos=outer north east,
    mark size=1.7pt
]
\addplot[mark=square*, thick, blue] table{
16  0.0924631083577 
32  0.0610739190534 
64  0.0391760034369 
128 0.0245307388874 
256 0.0150120554829 
512 0.00899235206971
};
\addlegendentry{\scriptsize $L^1$ error}
\addplot[mark=*, myred, thick] table {
16  0.021855878717  
32  0.0128721265976 
64  0.0076040556961 
128 0.00448638989534
256 0.00263195146715
512 0.00153105406619
};
\addlegendentry{\scriptsize $W_1$ error}
\addplot[domain=16:512, gray, dashed, thick]{(2/x)*log2(2/x)*(0.021855878717*16/(2*log2(2/16)))};
\addlegendentry{\scriptsize $O(\Dx|\log \Dx|)$}
\addplot[domain=16:512, mygreen, dashed, thick]{(2/x)*(0.021855878717*16/2)};
\addlegendentry{\scriptsize $O(\Dx)$}
\end{loglogaxis}
\end{tikzpicture}
\caption{Observed order of convergence in $L^1$ and $W_1$ compared with $O(\Dx|\log \Dx|)$ and $O(\Dx)$.}\label{fig: OOC upstep}
\end{figure}

\section{Concluding remarks}

In this paper we have shown optimality of the convergence rate $O(\Dx)$ in $W_1$ for monotone schemes in the case of $\Lip^+$-bounded initial data with compact support, and where the flux is assumed to be strictly convex. As noted in Table \ref{tab:survey of convergence rates} it is an open question whether the corresponding $L^1$ rate of $O(\Dx^\hf)$ is also optimal for this case since \c{S}abac's counter-example is $\Lip^+$-unbounded. Our numerical experiments (see Table \ref{tbl: Rates bump}) suggest that the counter-example considered here cannot be used to prove optimality of the rate $O(\Dx^\hf)$ in $L^1$ in this case. 

The convergence rate in $W_1$ for $\Lip^+$-unbounded initial data is still unknown. Our numerical test indicates that in the case of a rarefaction solution it could be the same as the $L^1$ rate, $O(\Dx|\log\Dx|)$. This is consistent with the rate $O(\epsilon|\log\epsilon|)$ proved in \cite{nessyahu1994convergence} for the viscous regularization of conservation laws with $\Lip^+$-unbounded initial data. Furthermore it can be heuristically explained by the same argument as in Section \ref{sec: Wasserstein} since the $L^1$ error in this case is $O(\Dx|\log\Dx|)$ \cite{harabetian1988rarefactions}.

Finally, to our knowledge there are currently no results on convergence rates in the Wasserstein distance for schemes for one-dimensional systems or for multidimensional conservation laws, although the $W_1$-distance can readily be defined in several dimensions.

\bibliographystyle{abbrv}

\end{document}